\documentclass{amsart}
\usepackage{graphicx, enumerate, amsmath, amssymb,amsthm,manfnt,hyperref,amscd,braket}
\usepackage[utf8]{inputenc} % set input encoding (not needed with XeLaTeX)
\allowdisplaybreaks

\newcommand{\rl}{{\mathbb{R}}}
\newcommand{\cx}{{\mathbb{C}}}

\newcommand{\Z}{\mathbb{Z}}

\theoremstyle{plain}

\newtheorem*{theorem*}{Theorem}
\newtheorem{theorem}{Theorem}[section]

\newtheorem{lemma}{Lemma}[section]

\theoremstyle{remark}
\newtheorem*{remark}{Remark}

\newcommand{\Bl}{\mathbf{B}_{\lambda}}
\newcommand{\Blt}{\mathbf{B}_{\tilde{\lambda}}}
\newcommand{\Ht}{\mathbb{H}}

\newcommand{\BpHmu}{\mathbf{B}^{\mu}_{\mathbb{H}}}

\newcommand{\Dd}{{\mathbb{D}^{\ast}\times\mathbb{D}}}

\newcommand{\Ud}{\mathbb{D}}

\newcommand{\BpOm}{\mathbf{B}_{\Omega}^{\mu}}
\newcommand{\BkOm}{B_{\Omega}^{\mu}(z,w)}
\newcommand{\BpOnd}{\mathbf{B}_{\Omega}^{\nu(\delta)}}
\newcommand{\BpOnD}{\mathbf{B}_{\mathbb{D}}^{\nu(\delta)}}
\newcommand{\BpHnd}{\mathbf{B}_{\mathbb{H}}^{\nu(\delta)}}
\newcommand{\BkHnd}{B_{\mathbb{H}}^{\nu(\delta)}}

\title[Weighted Bergman Projections on the Hartogs Triangle]{Weighted Bergman Projections on the Hartogs Triangle: Exponential Decay}
\author{Liwei Chen}
\address[Liwei Chen]{The Ohio State University, Department of
Mathematics, Columbus, OH 43210}
\email{chen.1690@osu.edu}

\author{Yunus E. Zeytuncu}
\address[Yunus E. Zeytuncu]{University of Michigan - Dearborn, Department of Mathematics and Statistics, Dearborn, MI 48128}
\email{zeytuncu@umich.edu}

\thanks{The work of the second author was partially supported by a grant from 
the Simons Foundation (\#353525).}
\begin{document}

\maketitle
\begin{abstract}
We study weighted Bergman projections on the Hartogs triangle in $\mathbb{C}^2$. We show that projections corresponding to exponentially vanishing weights have degenerate $L^p$ mapping properties. 
\end{abstract}

\section{Introduction}

Let $\Omega\subset\mathbb{C}^n$ be a bounded domain and $\mu$ be a non-negative function on $\Omega$. We say $\mu$ is an admissible weight if $L^2_a(\Omega, \mu)$, the space of square integrable holomorphic functions, is a closed subspace of $L^2(\Omega,\mu)$, the space of square integrable functions  with respect to $\mu(z)dV(z)$ where $dV(z)$ stands for the Lebesgue measure. The weighted Bergman projection $\BpOm$ is the orthogonal projection operator from $L^2(\Omega, \mu)$ onto $L^2_a(\Omega,\mu)$. It is an integral operator of the form 
$$\BpOm f(z)=\int_{\Omega}\BkOm f(w)\mu(w)dV(w)$$
for $f\in L^2(\Omega,\mu)$. We refer to \cite{KrantzSCVBook} and \cite{Pasternak90} for basic definitions and properties.
The analytic properties of the operator $\BpOm$ and kernel $\BkOm$ depend on the geometry of the domain $\Omega$ and the function theoretic properties of the weight $\mu$.

One particular investigation is relating the $L^p$ mapping properties of $\BpOm$ and the order of vanishing of the weight $\mu$ on the boundary of $\Omega$. One way of defining weights that vanish on the boundary is to take different compositions of a distance to the boundary function. In particular, let $\delta$ be a distance to the boundary function and $\nu(x): [0,\infty)\to[0,\infty)$ be a continuous function that only vanishes at $x=0$. Then the composition $\nu(\delta)$ is an admissible weight on $\Omega$ and one can study the $L^p$ regularity of $\BpOnd$ and relate it to the order of vanishing of $\nu(x)$ at $0$. Below we go over some known results using this notation. 

We start with the unit disc $\mathbb{D}$ in $\mathbb{C}$ and choose $\delta(z)=1-|z|^2$. Then we list the following results.
\begin{itemize}
\item If $\nu\equiv 1$ or $\nu(x)=x^k$ for some $k>0$, then the weighted projection operator $\BpOnD$ is bounded from $L^p(\mathbb{D}, \nu(\delta))$ to itself for all $p\in(1,\infty)$. See \cite{Zhu, ForelliRudin}.
\item On the other hand, if $\nu(x)=\exp\left(-\frac{1}{x}\right)$ then the  weighted projection operator $\BpOnD$ is bounded from $L^p(\mathbb{D}, \nu\left(\delta\right))$ to itself only for $p=2$. See \cite{Dostanic04, ZeytuncuTran}.
\end{itemize}

This change in $L^p$ mapping properties between the polynomial vanishing and the exponential vanishing has been detected on some Reinhardt domains too. In particular, let $\Omega$ be a smooth bounded complete Reinhardt domain of finite type in $\mathbb{C}^2$ and let $\rho$ be a smooth multi-radial defining function for $\Omega$. We choose $\delta=-\rho$. Then we recall the following results.
\begin{itemize}
\item If $\nu\equiv 1$ or $\nu(x)=x^t$ for some rational number $t>0$, then the weighted projection operator $\BpOnd$ is bounded from $L^p(\Omega, \nu(\delta))$ to itself for all $p\in(1,\infty)$. See \cite{McNeal94, Charpentier13, Charpentier14}.
\item On the other hand, if $\nu(x)=\exp\left(-\frac{1}{x}\right)$ then the  weighted projection operator $\BpOnd$ is bounded from $L^p(\Omega, \nu(\delta))$ to itself only for $p=2$. See \cite{CucZey}.
\end{itemize}

In addition to these contrasting results, there are more results in the polynomial decay case where the $L^p$ boundedness on the full interval $(1,\infty)$ is observed. See for example \cite{BonamiG} and \cite{ChangLi}. Therefore, the degenerate $L^p$ regularity for exponential weights on other domains arises as a natural question. More specifically, we pose the following question.\\

\textbf{Question:} Let $\Omega$ be a bounded domain with some additional geometric properties (e.g. finite type, convex, etc.) and $\delta$ be a distance to the boundary function. For $\nu(x)=\exp\left(-\frac{1}{x}\right)$, is the  weighted projection operator $\BpOnd$ bounded on $L^p(\Omega, \nu\left(\delta\right))$ for any $p\not=2$?\\

In this note, we investigate this question on the Hartogs triangle. In $\mathbb{C}^2$, the Hartogs triangle $\Ht$ is given by
\begin{equation*}
\Ht=\left\{ (z_1,z_2)\in \cx^2:~|z_2|<|z_1|<1\right\}.
\end{equation*}
The source of many counterexamples on $\Ht$ is the singular point at the origin. Hence, the natural choice of a distance function here is the distance to the singular point, that is, we set  \
$$\delta(z)=|z_1|.$$ 
For this choice of $\delta$, the polynomial decay case is already studied in \cite{Chen13, Chen14, ChakZey}. Although the $L^p$ boundedness does not hold for all $p\in(1,\infty)$, there is always a non-degenerate interval around $p=2$ where the weighted projection operator is $L^p$ bounded.

\begin{theorem*}[\cite{Chen14}] Let $\delta(z)=|z_1|$ and $\nu(x)=x^{t}$ where $t\in(0,\infty)$ with the unique decomposition $t=s+2k$ such that $k\in\mathbb{Z}$ and $s\in(0,2]$. Then the weighted Bergman projection $\BpHnd$ is bounded on $L^p(\Ht,\nu(\delta))$ if and only if $p\in\left(\frac{t+4}{s+k+1},\frac{t+4}{k+1}\right)$.
\end{theorem*}

Note that for any $t>0$, there is an interval around $2$ where the weighted Bergman projection is $L^p$ bounded. Moreover, this interval gets smaller as $t$ gets bigger. In other words, higher order vanishing of $\nu(x)$ indicates smaller $L^p$ boundedness range. In the light of this observation we answer the question above as follows.

\begin{theorem}\label{main}
Let $\nu(x)=\exp\left(-\frac{1}{x}\right)$, $\delta(z)=|z_1|$, and define the exponentially decaying weight
\begin{equation*}
\mu(z)=\nu(\delta(z))=\exp\left(-\frac{1}{|z_1|}\right).
\end{equation*}
Then the weighted Bergman projection $\BpHmu$ is bounded on $L^p(\Ht,\mu)$ if and only if $p=2$.
\end{theorem}

Although we state the result for a single choice of $\nu(x)$, it will be clear in the proof that the statement holds for more general choices of the form $\nu(x)=x^s\exp\left(-\frac{1}{x}\right)$ for $s\in\mathbb{R}$. Moreover, further generalizations can be formulated on domains that are variants of the Hartogs triangle, see \cite{Chen13} and \cite{EdholmMcNeal15}.\\

In addition to the change in $L^p$ regularity of the weighted Bergman projection, another related change takes place in the boundary behavior of the weighted Bergman kernel with exponentially decaying weights. For polynomially decaying weights on the Hartogs triangle, the weighted Bergman kernel on the diagonal does not grow faster than a power of the distance function. The following result can be derived from \cite[Lemma 3.1]{Chen14}, where the author presents a detailed study of polynomial weights on the Hartogs triangle.

\begin{theorem*}[\cite{Chen14}] Let $\delta(z)=|z_1|$ and $\nu(x)=x^{t}$ where $t\in(0,\infty)$ with the unique decomposition $t=s+2k$ such that $k\in\mathbb{Z}$ and $s\in(0,2]$. Then there exist real numbers $\tau$ and $C$ such that 
\begin{equation*}
\delta(z)^{\tau}\BkHnd(z,z)\leq C
\end{equation*}
as $z$ approaches the origin inside any cone $V_{\gamma}=\{(z_1,z_2)\in\mathbb{C}^2:~\gamma |z_2|<|z_1|\}$, where $\gamma>1$.
\end{theorem*}

However, the weighted Bergman kernel with respect to exponential weights grows faster than any power of the distance to the boundary function. We state the following result.

\begin{theorem}\label{kernel}
Let $\nu(x)=\exp\left(-\frac{1}{x}\right)$, $\delta(z)=|z_1|$, and define the exponentially decaying weight
\begin{equation*}
\mu(z)=\nu(\delta(z))=\exp\left(-\frac{1}{|z_1|}\right).
\end{equation*}
Then for any $\tau>0$,
\begin{equation*}
\delta(z)^{\tau}\BkHnd(z,z)
\end{equation*}
is unbounded as $z$ approaches the origin inside any cone $V_{\gamma}=\{(z_1,z_2)\in\mathbb{C}^2:~\gamma |z_2|<|z_1|\}$, where $\gamma>1$.
\end{theorem}
Furthermore, if we insert a factor of $\mu(z)$ into the product then we get boundedness. See the remark after the proof of Theorem \ref{kernel}.

Finally, we demonstrate another change from polynomial decay to exponential decay in the setting of generalized Hartogs triangles. Recently, in a series of papers \cite{EdholmMcNeal15,Edholm15}, $L^p$ boundedness of the Bergman projection operator is studied on the following variants of the Hartogs triangle. Let $k$ be a positive number and define
\begin{equation*}
\mathbb{H}_{k}=\{(z_1,z_2)\in\cx^2~:~|z_2|<|z_1|^{k}<1\}\footnote{In \cite{EdholmMcNeal15,Edholm15}, authors denote this domain by $\mathbb{H}_{\frac{1}{k}}$.}.
\end{equation*}
It is shown that when $k$ is a positive integer, the (unweighted) Bergman projection is bounded on $L^p(\mathbb{H}_k)$ for some values of $p$.

\begin{theorem*}[\cite{EdholmMcNeal15, EdhMcN}] For positive integer $k$, the Bergman projection $\mathbf{B}_{\mathbb{H}_k}$ is bounded on $L^p(\mathbb{H}_k)$ if and only if $p\in\left(\frac{2k+2}{k+2},\frac{2k+2}{k}\right)$.
\end{theorem*}

On the other hand, the domain $\mathbb{H}_k$ will become pieces of boundaries as $k$ tends to $\infty$. In order to study the degenerate $L^p$ boundedness, we define the following exponential version as the limiting domain in $\cx^2$
\begin{equation*}
\mathbb{H}_{\infty}=\left\{(z_1,z_2)\in\cx^2:~|z_2|<\exp\left(-\frac{1}{|z_1|}\right),~0<|z_1|<1\right\}.
\end{equation*}
Then we prove the following result.

\begin{theorem}
\label{infinity}
The Bergman projection $\mathbf{B}_{\mathbb{H}_{\infty}}$ is bounded on $L^p(\mathbb{H}_{\infty})$ if and only if $p=2$.
\end{theorem}

Recently, Edholm and McNeal \cite{EdhMcN} obtained similar degenerate $L^p$ regularity on domains $\mathbb{H}_{\gamma}$ where $\gamma$ is an irrational number. These examples are particularly interesting since the degeneracy is not due to the exponential decay but to non-rationality of the singularity.\\

In the following sections, we present proofs for the three statements above. We start with Theorem \ref{main}, where the proof is based on some asymptotic computations. Then we continue with Theorem \ref{kernel} and we present a proof by obtaining an (almost) explicit closed form for the weighted Bergman kernel on the diagonal. Finally, the proof of Theorem \ref{infinity} is based on the weighted theory on the punctured disc.

We highlight again that the \textit{irregular} behavior in Theorems \ref{main} and \ref{kernel} are due to the exponential decay of the weight and similar results may hold on other domains. We plan to study these analogous questions on more general domains in future.

Throughout the paper, we write $x\approx y$ to mean that there exists $C>0$ such that $\frac{1}{C}x\leq y\leq Cx$.\\

%%%%%%%%%%%%%%%%%%%%%%%%%%%%%%
%%%%%%%%%%%%%%%%%%%%%%%%%%%%%%
%%%%%%%%%%%%%%%%%%%%%%%%%%%%%%

\section{Proof of Theorem \ref{main}}

We prove Theorem \ref{main} by studying weighted Bergman projections on the punctured unit disc $\mathbb{D}^*$. Let $\Bl$ be the weighted Bergman projection from $L^2(\mathbb{D}^*, \lambda)$ onto $L_a^2(\mathbb{D}^*, \lambda)$, where $\lambda(z)=\exp\left(-\frac{1}{|z|}\right)$, $z\in\mathbb{D}^*$. Given any sufficiently large $j\in\Z^+$, we pick $p>2+\frac{2}{j}$. We study the behavior of the sequence $\left\{z^{-jk}\overline{z}^k\right\}_{k\in\Z^+}$ under the projection $\Bl$. 

First, for $\alpha\in\rl$, we define
\begin{equation*}
I(\alpha)=\int_0^1r^{\alpha}e^{-\frac{1}{r}}\,dr.
\end{equation*}
Since $\lambda$ is radial on $\mathbb{D}^*$ and exponential decaying at the origin, all $\{z^n\}_{n\in\Z}$ are orthogonal and in $L_a^2(\mathbb{D}^*, \lambda)$. So the weighted Bergman kernel has the form
\begin{equation*}
B_{\lambda}(z,\zeta)=\sum_{n=-\infty}^{\infty}c_n(z\overline{\zeta})^n,
\end{equation*}
where $c_n=\Big(\int_{\mathbb{D}^*}|z|^{2n}\exp\big(-\frac{1}{|z|}\big)\,dA(z)\Big)^{-1}$. Then by using the orthogonality of monomials and the labelling above, we get
\begin{equation}
\label{test}
\begin{split}
\frac{\|\Bl(z^{-jk}\overline{z}^k)\|^p_{\lambda}}{\|z^{-jk}\overline{z}^k\|^p_{\lambda}}
&=\frac{\int_{\mathbb{D}^*}\left|\int_{\mathbb{D}^*}B_{\lambda}(z,\zeta)\cdot\zeta^{-jk}\overline{\zeta}^k\cdot\exp\Big(-\frac{1}{|\zeta|}\Big)\,dA(\zeta)\right|^p\exp\Big(-\frac{1}{|z|}\Big)\,dA(z)}{\int_{\mathbb{D}^*}|z^{-jk}\overline{z}^{k}|^p\exp\Big(-\frac{1}{|z|}\Big)\,dA(z)}\\
&=\frac{\int_{\mathbb{D}^*}\left|\int_{\mathbb{D}^*}\sum_{n=-\infty}^{\infty}c_n(z\overline{\zeta})^n\cdot\zeta^{-jk}\overline{\zeta}^k\cdot\exp\Big(-\frac{1}{|\zeta|}\Big)\,dA(\zeta)\right|^p\exp\Big(-\frac{1}{|z|}\Big)\,dA(z)}{\int_{\mathbb{D}^*}|z|^{-(j-1)pk}\exp\Big(-\frac{1}{|z|}\Big)\,dA(z)}\\
&=\frac{\left|\int_{\mathbb{D}^*}|\overline{\zeta}|^{-2jk}\cdot\exp\Big(-\frac{1}{|\zeta|}\Big)\,dA(\zeta)\right|^p\int_{\mathbb{D}^*}|z|^{-(j+1)pk}\exp\Big(-\frac{1}{|z|}\Big)\,dA(z)}{\left|\int_{\mathbb{D}^*}|z|^{-2(j+1)k}\exp\Big(-\frac{1}{|z|}\Big)\,dA(z)\right|^p\int_{\mathbb{D}^*}|z|^{-(j-1)pk}\exp\Big(-\frac{1}{|z|}\Big)\,dA(z)}\\
&=\frac{\left[I(-2jk+1)\right]^p\cdot I(-(j+1)pk+1)}{[I(-2(j+1)k+1)]^p\cdot I(-(j-1)pk+1)}.
\end{split}
\end{equation}
Our goal is to show that the fraction in the last line blows up. We accomplish this by studying the asymptotic behavior of the integral $I(\alpha)$.

\begin{lemma}
We have the following estimates on $I(\alpha)$ as $\alpha\to\pm\infty$.
\begin{equation}
\label{pos}
\lim_{\alpha\to+\infty}(\alpha+1) I(\alpha)=e^{-1}
\end{equation}

\begin{equation}
\label{neg}
\lim_{\alpha\to+\infty}\frac{I(-\alpha)}{\Gamma(\alpha-1)}=1
\end{equation}
\end{lemma}

\begin{proof}
To show \eqref{pos}, we first apply the Monotone Convergent Theorem and conclude
\begin{equation*}
\lim_{\alpha\to+\infty}I(\alpha)=0.
\end{equation*}
By integration by parts, we see that
\begin{align*}
I(\alpha)&=\int_0^1r^{\alpha}e^{-\frac{1}{r}}\,dr\\
&=\frac{1}{\alpha+1}r^{\alpha+1}e^{-\frac{1}{r}}\Bigg|_{0}^{1}-\int_{0}^{1}\frac{1}{\alpha+1}r^{\alpha-1}e^{-\frac{1}{r}}\,dr\\
&=\frac{e^{-1}}{\alpha+1}-\frac{I(\alpha-1)}{\alpha+1}.
\end{align*}
By clearing the denominators, and letting $\alpha\to+\infty$, we arrive at \eqref{pos}.

To show \eqref{neg}, we make change of variables in the definition of $I(\alpha)$ by $x=\frac{1}{r}$. We see that
\begin{align*}
I(-\alpha)
&=\int_1^{\infty}x^{\alpha-2}e^{-x}\,dx\\
&=\Gamma(\alpha-1)-\int_0^1x^{\alpha-2}e^{-x}\,dx,
\end{align*}
where $\Gamma$ is the Gamma function. Again by the Monotone Convergent Theorem
\begin{equation*}
\lim_{\alpha\to+\infty}\int_0^1x^{\alpha-2}e^{-x}\,dx=0,
\end{equation*}
and the fact that
\begin{equation*}
\lim_{\alpha\to+\infty}\Gamma(\alpha-1)=+\infty,
\end{equation*}
we obtain \eqref{neg}.
\end{proof}
\medskip

Now we combinte the asymptotic estimate in \eqref{neg} and the Stirling's formula
\begin{equation*}
\Gamma(x+1)\approx \sqrt{2\pi x}\left(\frac{x}{e}\right)^x
\end{equation*}
as $x\to\infty$, to get
\begin{align*}
\lim_{k\to+\infty}\frac{\|\Bl(z^{-jk}\overline{z}^k)\|^p_{\lambda}}{\|z^{-jk}\overline{z}^k\|^p_{\lambda}}
&=\lim_{k\to+\infty}\frac{[\Gamma(2jk-2)]^p\cdot\Gamma((j+1)pk-2)}{[\Gamma(2(j+1)k-2)]^p\cdot\Gamma((j-1)pk-2)}\\
&=\lim_{k\to+\infty}\sqrt{\left(\frac{j}{j+1}\right)^p\left(\frac{j+1}{j-1}\right)}\cdot\frac{(2jk-3)^{(2jk-3)p}[(j+1)pk-3]^{[(j+1)pk-3]}}{[2(j+1)k-3]^{[2(j+1)k-3]p}[(j-1)pk-3]^{[(j-1)pk-3]}}\\
&\approx\lim_{k\to+\infty}\left(\frac{2jk-3}{2(j+1)k-3}\right)^{(2jk-3)p}\cdot\left(\frac{(j+1)pk-3}{(j-1)pk-3}\right)^{(j-1)pk-3}\cdot\left(\frac{(j+1)pk-3}{2(j+1)k-3}\right)^{2pk}.
\end{align*}
\medskip

A straightforward limit computation indicates 
\begin{align*}
\lim_{k\to+\infty}\left(\frac{2jk-3}{2(j+1)k-3}\right)^{(2jk-3)p}\cdot\left(\frac{j+1}{j}\right)^{(2jk-3)p}
&=\lim_{k\to+\infty}\left(1-\frac{3}{2j(j+1)k-3j}\right)^{(2jk-3)p}\\
&=\exp\left(-\frac{3p}{j+1}\right).
\end{align*}
Similarly, we also get
\begin{align*}
\lim_{k\to+\infty}\left(\frac{(j+1)pk-3}{(j-1)pk-3}\right)^{(j-1)pk-3}\cdot\left(\frac{j-1}{j+1}\right)^{(j-1)pk-3}
&=\lim_{k\to+\infty}\left(1+\frac{6}{(j+1)(j-1)pk-3(j+1)}\right)^{(j-1)pk-3}\\
&=\exp\left(\frac{6}{j+1}\right)
\end{align*}
and
\begin{align*}
\lim_{k\to+\infty}\left(\frac{(j+1)pk-3}{2(j+1)k-3}\right)^{2pk}\cdot\left(\frac{2}{p}\right)^{2pk}
&=\lim_{k\to+\infty}\left(1+\frac{3p-6}{2(j+1)pk-3p}\right)^{2pk}\\
&=\exp\left(\frac{3p-6}{j+1}\right).
\end{align*}
Therefore, since $\frac{j+1}{j-1}\ge\left(\frac{j+1}{j}\right)^2$, we obtain
\begin{align*}
\lim_{k\to+\infty}\frac{\|\Bl(z^{-jk}\overline{z}^k)\|^p_{\lambda}}{\|z^{-jk}\overline{z}^k\|^p_{\lambda}}
&\approx\lim_{k\to+\infty}\left(\frac{j}{j+1}\right)^{(2jk-3)p}\left(\frac{j+1}{j-1}\right)^{(j-1)pk-3}\left(\frac{p}{2}\right)^{2pk}\\
&\ge\lim_{k\to+\infty}\left(\frac{j+1}{j}\right)^{2(j-1)pk-6-(2jk-3)p}\left(\frac{p}{2}\right)^{2pk}\\
&=\lim_{k\to+\infty}\left(\frac{j+1}{j}\right)^{3p-6-2pk}\left(\frac{p}{2}\right)^{2pk}\\
&\ge\lim_{k\to+\infty}\left(1+\frac{1}{j}\right)^{-2pk}\left(\frac{p}{2}\right)^{2pk}\\
&=\infty.
\end{align*}
This shows that for any $p>2$, the weighted Bergman projection $\Bl$ is unbounded on $L^p(\mathbb{D}^*, \lambda)$.
\medskip

Now, we deduce the unboundedness of $\BpHmu$ from unboundedness of the weighted Bergman projections as follows. 
Let $\Blt$ be the weighted Bergman projection on $L_a^2(\mathbb{D}^*, \tilde{\lambda})$, where $\tilde{\lambda}(z)=|z|^2\exp\left(-\frac{1}{|z|}\right)$, $z\in\mathbb{D}^*$. Then by the inflation principle (see \cite{ZeytuncuTran, Chen14}), if $\BpHmu$ is bounded on $L^p(\Ht,\mu)$ then $\Blt$ is bounded on $L^p(\mathbb{D}^*, \tilde{\lambda})$.

Note that the behavior of the sequence $\left\{z^{-jk}\overline{z}^k\right\}_{k\in\Z^+}$ under the projection $\Blt$ can be obtained from the behavior of the same sequence under the projection $\Bl$ by shifting the indices of the integral $I(\alpha)$ by $2$. More precisely, we have

\begin{equation*}
\frac{\|\Blt(z^{-jk}\overline{z}^k)\|^p_{\tilde{\lambda}}}{\|z^{-jk}\overline{z}^k\|^p_{\tilde{\lambda}}}=\frac{[I(-2jk+3)]^p\cdot I(-(j+1)pk+3)}{[I(-2(j+1)k+3)]^p\cdot I(-(j-1)pk+3)}
\end{equation*}
and
\begin{equation*}
\lim_{k\to+\infty}\frac{\|\Blt(z^{-jk}\overline{z}^k)\|^p_{\tilde{\lambda}}}{\|z^{-jk}\overline{z}^k\|^p_{\tilde{\lambda}}}\cdot\left(\frac{j}{j+1}\right)^{2p}\left(\frac{j+1}{j-1}\right)^2=\lim_{k\to+\infty}\frac{\|\Bl(z^{-jk}\overline{z}^k)\|^p_{\lambda}}{\|z^{-jk}\overline{z}^k\|^p_{\lambda}}.
\end{equation*}
\medskip 

Hence, for any $p>2$ the weighted Bergman projection $\Blt$ is unbounded on $L^p(\mathbb{D}^*,\tilde{\lambda})$ and we conclude that $\BpHmu$ is bounded on $L^p(\Ht,\mu)$ if and only if $p=2$.
\medskip

%%%%%%%%%%%%%%%%%%%%%%%%%%%%%%
%%%%%%%%%%%%%%%%%%%%%%%%%%%%%%
%%%%%%%%%%%%%%%%%%%%%%%%%%%%%%

\section{Proof of Theorem \ref{kernel}}

In this proof, we again study the weighted Bergman space $L_a^2(\mathbb{D}^*,\lambda)$, where $\lambda(z)=\exp\left(-\frac{1}{|z|}\right)$. Note that, if we denote the weighted Bergman kernel associated to $L_a^2(\mathbb{D}^*,\lambda)$ by $B_{\lambda}(z,\zeta)$, then we have
\begin{equation*}
B_{\lambda}(z,\zeta)=\sum_{k=-\infty}^{+\infty}\frac{1}{I(2k+1)}z^k\overline{\zeta}^k.
\end{equation*}
By \eqref{pos} and \eqref{neg}, we see that
\begin{align*}
B_{\lambda}(z,z)
&=\sum_{k=-\infty}^{-1}+\sum_{k=0}^{+\infty}\frac{1}{I(2k+1)}|z|^{2k}\\
&\approx\frac{|z|^{-2}}{I(-1)}+\sum_{k=2}^{+\infty}\frac{|z|^{-2k}}{\Gamma(2k-2)}+\sum_{k=0}^{+\infty}(k+1)|z|^{2k}\\
&\approx|z|^{-3}\sinh|z|^{-1}+\frac{1}{(1-|z|^2)^2}.
\end{align*}

Let $B_{\Ht}^{\mu}(z,\zeta)$ be the Bergman kernel on $\Ht$ with respect to the weight $\mu(z_1,z_2)=\exp\left(-\frac{1}{|z_1|}\right)$, and let $B_{\Ht}^{\mu'}(z,\zeta)$ be the Bergman kernel on $\Ht$ with respect to the weight $\mu'(z_1,z_2)=|z_1|^{t}$, where $t\in\rl$. By looking at the biholomorphism

\[
\Phi:\Ht\to\mathbb{D}^*\times\mathbb{D},~~\text{ where  }~~\Phi(z_1,z_2)=\left(z_1,\frac{z_2}{z_1}\right)
\]
and by the biholomorphic equivalence of kernels, see \cite[Corollary 2.4 and Lemma 3.1]{Chen14}, we conclude that

\begin{align*}
B_{\Ht}^{\mu'}(z,\zeta)
&=\det J_{\cx}\Phi(z)\overline{\det J_{\cx}\Phi(\zeta)}B_{\Dd}^{\mu'\circ\Phi^{-1}}(\Phi(z),\Phi(\zeta))\\
&=\frac{1}{z_1}\cdot\frac{1}{\overline{\zeta}_1}\cdot B_{\mathbb{D}^*}^{\mu'\circ\Phi^{-1}}(z_1,\zeta_1)\cdot B_{\Ud}\left(\frac{z_2}{z_1},\frac{\zeta_2}{\zeta_1}\right)\\
&=\left[\frac{s}{2}\cdot\frac{1}{(z_1\overline{\zeta}_1)^{k+2}}+\left(1-\frac{s}{2}\right)\cdot\frac{1}{(z_1\overline{\zeta}_1)^{k+1}}\right]\cdot\frac{1}{(1-z_1\overline{\zeta}_1)^2}\cdot\frac{1}{\left(1-\frac{z_2}{z_1}\cdot\frac{\overline{\zeta}_2}{\overline{\zeta}_1}\right)^2}
\end{align*}
where $t=s+2k$, $k\in\Z$ and $s\in(0,2]$, and that
\begin{align*}
B_{\Ht}^{\mu}(z,\zeta)
&=\det J_{\cx}\Phi(z)\overline{\det J_{\cx}\Phi(\zeta)}B_{\Dd}^{\mu\circ\Phi^{-1}}(\Phi(z),\Phi(\zeta))\\
&=\frac{1}{z_1}\cdot\frac{1}{\overline{\zeta}_1}\cdot B_{\mathbb{D}^*}^{\mu\circ\Phi^{-1}}(z_1,\zeta_1)\cdot B_{\Ud}\left(\frac{z_2}{z_1},\frac{\zeta_2}{\zeta_1}\right)\\
&=B_{\lambda}(z_1,\zeta_1)\cdot\frac{1}{z_1}\cdot\frac{1}{\overline{\zeta}_1}\cdot\frac{1}{\left(1-\frac{z_2}{z_1}\cdot\frac{\overline{\zeta}_2}{\overline{\zeta}_1}\right)^2}.
\end{align*}

Therefore, on the diagonal line, we have
\begin{equation*}
B_{\Ht}^{\mu'}(z,z)=\left[\frac{s}{2}\cdot\frac{1}{|z_1|^{2k+4}}+\left(1-\frac{s}{2}\right)\cdot\frac{1}{|z_1|^{2k+2}}\right]\cdot\frac{1}{(1-|z_1|^2)^2}\cdot\frac{1}{\left(1-\frac{|z_2|^2}{|z_1|^2}\right)^2}
\end{equation*}
and
\begin{equation*}
B_{\Ht}^{\mu}(z,z)\approx\left[|z_1|^{-3}\sinh|z_1|^{-1}+\frac{1}{(1-|z_1|^2)^2}\right]\cdot\frac{1}{|z_1|^2}\cdot\frac{1}{\left(1-\frac{|z_2|^2}{|z_1|^2}\right)^2}.
\end{equation*}

Recall that $\delta(z)=|z_1|$, hence
\begin{align*}
B_{\Ht}^{\mu'}(z,z)\cdot\delta(z)^{\tau}
&=\left[\frac{s}{2}\cdot\frac{1}{|z_1|^{2k+4}}+\left(1-\frac{s}{2}\right)\cdot\frac{1}{|z_1|^{2k+2}}\right]\cdot\frac{1}{(1-|z_1|^2)^2}\cdot\frac{1}{\left(1-\frac{|z_2|^2}{|z_1|^2}\right)^2}\cdot|z_1|^{\tau}\\
&\le |z_1|^{\tau-(2k+4)}\cdot\frac{1}{(1-|z_1|^2)^2}\cdot\frac{1}{\left(1-\frac{|z_2|^2}{|z_1|^2}\right)^2}\\
&\le C
\end{align*}
for some $\tau\in\rl$ as $z\to0$ inside any cone $V_{\gamma}=\{(z_1,z_2)\in\mathbb{C}^2:~\gamma |z_2|<|z_1|\}$, where $\gamma>1$. Whereas, for the exponential weight $\mu$ on $\Ht$,
\begin{align*}
B_{\Ht}^{\mu}(z,z)\cdot\delta(z)^{\tau}
&\approx\left[|z_1|^{-3}\sinh|z_1|^{-1}+\frac{1}{(1-|z_1|^2)^2}\right]\cdot\frac{1}{|z_1|^2}\cdot\frac{1}{\left(1-\frac{|z_2|^2}{|z_1|^2}\right)^2}\cdot|z_1|^{\tau}\\
&\approx|z_1|^{\tau-5}\exp\left(\frac{1}{|z_1|}\right)\cdot\frac{1}{\left(1-\frac{|z_2|^2}{|z_1|^2}\right)^2}
\end{align*}
is unbounded for all $\tau\in\rl$ as $z\to0$ inside any cone $V_{\gamma}$, where $\gamma>1$.

\begin{remark}
If we correct the distance $\delta(z)$ by a factor of the weight $\mu(z)$, then we get
\begin{equation*}
B_{\Ht}^{\mu}(z,z)\cdot\mu(z)\cdot\delta(z)^{\tau}\approx|z_1|^{\tau-5}\cdot\left(1-\frac{|z_2|^2}{|z_1|^2}\right)^{-2}\le C
\end{equation*}
for $\tau\ge5$ as $z\to0$ inside any cone $V_{\gamma}$, where $\gamma>1$.
\end{remark}

%%%%%%%%%%%%%%%%%%%%%%%%%%%%%%
%%%%%%%%%%%%%%%%%%%%%%%%%%%%%%
%%%%%%%%%%%%%%%%%%%%%%%%%%%%%%

\section{Proof of Theorem \ref{infinity}}

In view of the Forelli-Rudin inflation principle, see for example \cite[Proposition 4.4]{ZeytuncuTran}, unboundedness of  $\mathbf{B}_{\mathbb{H}_{\infty}}$ can be deduced from unboundedness of the corresponding weighted Bergman projection on the punctured disc $\mathbb{D}^*$. Note that, in \cite{ZeytuncuTran}, the smooth radial weights vanish at infinite order at the boundary of the unit disc. Hence, integration by parts plays a crucial role in obtaining asymptotics of the moment function of these weights. Whereas, here the radial weight $\lambda(z)=\exp\left(-\frac{1}{|z|}\right)$ is vanishing of any order only at the origin, the non-smooth boundary of the punctured disc $\mathbb{D}^*$. Since we do not have vanishing conditions on the smooth boundary of the punctured disc, we do not base our argument on successive integration by parts. Instead we use the asymptotic information from the second section.

\begin{proof}[Proof of Theorem \ref{infinity}]
By \cite[Proposition 4.4]{ZeytuncuTran}, for $p\in(1,\infty)$ if the weighted Bergman projection $\Bl$ is unbounded on the weighted space $L^p(\mathbb{D}^*,\lambda)$, then the Bergman projection $\mathbf{B}_{\mathbb{H}_{\infty}}$ associated to $\mathbb{H}_{\infty}$ is unbounded on $L^p(\mathbb{H}_{\infty})$. In Section 2, we have already showed that $\Bl$ is unbounded on $L^p(\mathbb{D}^*,\lambda)$ for $p>2$, therefore we conclude that $\mathbf{B}_{\mathbb{H}_{\infty}}$ is bounded on $L^p(\mathbb{H}_{\infty})$ if and only if $p=2$.
\end{proof}

%\bibliographystyle{alpha}
%\bibliography{Lp}

\end{document}